\documentclass[english,12pt, a4paper,reqno,oneside]{article}

\usepackage{amsfonts,	amsmath,array}
\usepackage{bbm}
\usepackage{babel}
\usepackage{IEEEtrantools}

\usepackage{amssymb,latexsym}
\usepackage{graphicx,amsfonts}
\usepackage{epsfig}
\usepackage{pstricks}
\usepackage{pst-node}
\usepackage{amsthm}

\newcommand{\bb}{\mathbb}

\newcommand{\R}{\bb R}

\newcommand{\N}{\bb N}

\renewcommand{\epsilon}{\varepsilon}

\newtheorem{claim}{Claim}

\newtheorem{theorem}{Theorem}
\newtheorem{lemma}{Lemma}
\newtheorem{definition}{Definition}

\begin{document}

\begin{title}
{On the existence of compact $\epsilon$-approximated formulations for knapsack in the original space}
\end{title}

\author{Yuri Faenza\thanks{DISOPT, Ecole Polytechnique F\'ed\'erale de Lausanne (Switzerland): \texttt{yuri.faenza@epfl.ch}. Supported by the \emph{Ambizione} grant PZ00P2$\_$154779 \emph{Tight formulations of 0-1 problems} funded by the Swiss National Science Foundation.}  \and Laura Sanit\`a\thanks{Combinatorics and Optimization department, University of Waterloo (Canada): \texttt{lsanita@uwaterloo.ca}}}

\date{}
\maketitle

\abstract{We show that there exists a family $\mathcal P$ of Knapsack polytopes such that for each $P \in \mathcal P$ and each $\epsilon >0$,  any $\epsilon$-approximated formulation of $P$ in the original space $\mathbb R^n$ requires a number of inequalities that is super-polynomial in $n$. This answers a question by Bienstock and McClosky (2012). 
We also prove that, for any down-monotone polytope, an $\epsilon$-approximated formulation in the original space can be obtained with inequalities using at most $O(\frac{1}{\epsilon}\min\{\log(n/\epsilon),n\})$ different coefficients.
}

\section{Introduction}

The \emph{Knapsack} problem is one of the classics of combinatorial optimization: given a set of $n$ objects, each coming with a nonnegative weight and profit, and a threshold $B$, find the most profitable subset of objects whose total weight does not exceed $B$. Knapsack has been extensively studied in the combinatorial optimization and integer programming communities (see e.g.~\cite{KePfPi,NeWo}).

From an algorithmic point of view, we know that the associated decision problem is NP-Complete~\cite{Ka}. On the sunny side, the Knapsack problem admits a \emph{fully polynomial-time approximation scheme} (FPTAS)~\cite{IbKi,La}. That is, for each $\epsilon \in \R_+$ and $n \in \N$, a feasible solution to a Knapsack instance on $n$ objects of cost at least $(1-\epsilon)$ the cost of the optimum can be found in time polynomial in $1/\epsilon$ and $n$. Roughly speaking, this implies that we can find an approximated solution with cost arbitrarily close to the one of the optimum, in time polynomial in the input size and in the accuracy of the approximation. This algorithm relies on classical combinatorial techniques such as scaling, rounding, and dynamic programming.

A natural question \cite{Bien,BieMcc,VVWol} is whether we can translate the \emph{algorithmic} approximation given by the FPTAS into  a \emph{polyhedral} approximation, that is,
whether we can approximate the \emph{Knapsack polytope} using a polynomial number of inequalities, where the Knapsack polytope is the convex hull of all $0-1$ vectors in $\mathbb R^n$ corresponding to feasible solutions of a given Knapsack instance. Let us formalize the concept of polyhedral approximation by introducing some definitions.

\begin{definition}
Given a polytope $P\subseteq \R^n$ and $\epsilon>0$, we say that $Q \subseteq \R^n$ is an \emph{$\varepsilon$-approximated formulation} for $P$ if  $P \subseteq Q$, and for each $c \in \R^n$ the following inequality holds \begin{equation}\label{eq:ratio} \max \{cx : x \in P\} \geq (1-\varepsilon) \max \{cx : x \in Q\}.\end{equation}
\end{definition}

For a polytope $Q\subseteq \R^n$, we denote by  \emph{size}$(Q)$ the number of inequalities in a minimum linear description of $Q$ (i.e. the number of \emph{facets} of $Q$). We also denote by \emph{xc}$(Q)$ its \emph{extension complexity}, that is the minimum size of a polytope $\widetilde Q$ such that there exists a linear map $\pi$ with $Q = \pi (\widetilde Q)$. Note that $Q$ and $\widetilde Q$ can have different dimensions.

A family ${\cal P}$ of
polytopes is said to have a \emph{polynomial-size relaxation schemes} (PSRS) if for each $\varepsilon > 0$ there exists a polynomial function $\phi: \N \rightarrow \N$ such that for each $P \in {\cal P}$ with $P \subseteq \R^n$, $P$ has an $\varepsilon$-approximated formulation $Q$ with \emph{xc}$(Q) \leq \phi(n)$. \\
If, in addition,  the function $\phi$ grows polynomially in $\frac{1}{\varepsilon}$, then ${\cal P}$ is said to have a \emph{fully polynomial-size relaxation schemes} (FPSRS).

The above definitions (see~\cite{BrPo}) of PSRS and  FPSRS can be seen as a translation of the algorithmic notions of FPTAS and PTAS into polyhedral terms.

\smallskip
For  the family of Knapsack polytopes, the existence of a PSRS has been proved by Bienstock~\cite{Bien}, while the existence of a FPSRS is still an interesting open question~\cite{VVWol,Bien}.

The PSRS provided by \cite{Bien} yields an $\varepsilon$-approximated formulation $Q$ with small extension complexity, but it does not imply strong bounds on size$(Q)$. In fact, in the same paper Bienstock~\cite{Bien}, as well as Bienstock and McClosky~\cite[Question 1, page 340]{BieMcc}, ask whether such result can be straightened by proving  the existence of 
a PSRS \emph{in the original space}, that is given by replacing \emph{xc}$(Q)$ with \emph{size}$(Q)$ in the definition of PSRS. In other words, whether for any Knapsack polytope $P \subseteq \mathbb R^n$ we can find  an $\epsilon$-approximated formulation $Q \subseteq \mathbb R^n$ with
a number of \emph{facets} bounded by a polynomial function $\phi(n)$.

In this paper, we negatively answer this open question.
\begin{theorem}
\label{thm:main}
The family of Knapsack polytopes admits no PSRS in the original space.
\end{theorem}

Bienstock and McClosky~\cite{BieMcc} also ask for upper bounds on the number of inequalities needed to obtain an $\epsilon$-approximation in the original space for the Knapsack polytope. To the best of our knowledge, the only known upper bound is the one of Van Vyve~\cite{VV}, cited in~\cite{BieMcc}, of order $(\lceil n/\epsilon\rceil+1)^n$. This follows from the fact that the formulation containing all valid inequalities with integer coefficients between $0$ and $\lceil n/\epsilon\rceil$ is $\epsilon$-approximated. In~\cite{BieMcc}, it is also shown that, for each $\gamma>0$, there exist Knapsack polytopes $P$ such that the polytope $Q$ given by all valid inequalities with integer coefficients up to $n^{1-\gamma}$ is a bad approximation. Namely, there exists an objective function $c \in \R_+$ such that the ratio \eqref{eq:ratio} is arbitrarily close to $1/2$.

Our second result is a new upper bound on the number of different coefficients (hence on the number of facets) needed to obtain an $\epsilon$-approximated formulation for Knapsack polytopes in the original space. Our upper bound improves over the bound in \cite{VV} whenever $1/\epsilon$ is a sub-exponential function of $n$, and asymptotically matches it otherwise. Moreover, our result holds for any \emph{down-monotone} polytope, that is, any polytope $P\subseteq \R^n_+$ with the property that $x \in P$ and $0 \leq y \leq x$, implies $y \in P$.

\begin{theorem}\label{thm:upper_bound}
For any down-monotone polytope $P \subseteq \R^n_+$, an $\epsilon$-approximated formulation in the original space can be obtained with inequalities using at most $O(\frac{1}{\epsilon}\min\{\log (n/\epsilon),n\})$ different coefficients. Each of those coefficients is an integer between $0$ and $5n(1-\epsilon)/\epsilon$. In particular, there exists an $\epsilon$-approximated formulation with $O(\frac{1}{\epsilon}\min\{\log (n/\epsilon),n\})^n$ facets.
\end{theorem}

We emphasize that Theorem \ref{thm:upper_bound} implies the existence of $\epsilon$-approximated formulations of a down-monotone polytope with only $O(\log n)$ different coefficients, if $\epsilon$ is \emph{fixed}.

\section{No PSRS in the original space exists for Knapsack}
The goal of this section is to prove Theorem \ref{thm:main}.

\paragraph{The strategy.} For any given $\varepsilon>0$ small enough, we provide an infinite sequence of integers, and to each integer $n$ from this sequence we associate a Knapsack polytope $P(n,\varepsilon)\subseteq \R^n$ and a family of points ${\cal X}(n,\varepsilon)\subseteq \R^n$ with three properties:
\begin{itemize}
\item[(i)] $|{\cal X}(n,\varepsilon)|$ is super-polynomial in $n$;
\item[(ii)] for each $\bar x \in {\cal X}(n,\varepsilon)$, there exists $c \in \R_+^n$ such that $(1-\varepsilon)c\bar x> \max\{cx : x \in P(n,\varepsilon)\}$;
\item[(iii)] for each pair of distinct points $\bar x_1, \bar x_2 \in {\cal X}(n,\varepsilon)$, the line segment between $\bar x_1$ and $\bar x_2$ intersects $P(n,\varepsilon)$.
\end{itemize}

In particular, (ii) implies that $P(n,\varepsilon)\cap {\cal X}(n,\varepsilon)= \emptyset$ and, most important, that any $\varepsilon$-approximation of $P(n,\varepsilon)$ must have an inequality separating each $\bar x \in {\cal X}(n,\varepsilon)$ from $P(n,\varepsilon)$. However, (iii) implies that no two points $\bar x_1, \bar x_2 \in {\cal X}(n,\varepsilon)$ can be separated by \emph{the same} inequality, since the line segment between them intersects $P(n,\varepsilon)$.
It follows that any description of an  $\varepsilon$-approximation of $P(n,\varepsilon)$  contains at least $|{\cal X}(n,\varepsilon)|$ distinct inequalities, that is super-polynomial in $n$ by (i). As $\epsilon$ was arbitrary and the choices of $n$ were infinite, this implies Theorem \ref{thm:main}. We remark that similar strategies have been recently used by Kaibel and Weltge~\cite{KaWe} and Kolliopoulos and Moysoglou~\cite{KoMo}.

\paragraph{Constructing a family of Knapsack polytopes.}
Let $\varepsilon >0$ be a fixed small number $< \frac{2}{27}$. For any  $n > 1/\varepsilon$, $n\geq 130$ that is the square of a prime number, we construct a Knapsack polytope as follows:

$$P(n+1,\varepsilon)=\{x \in \{0,1\}^{n+1}:  \sum_{i=1}^{n} \frac{1}{2\varepsilon \sqrt n} x_i  + n x_{n+1} \leq n + \frac{1}{2\varepsilon}-1 \},$$

Interestingly, note that $P(n+1,\varepsilon)$ describes a somewhat ``easy" Knapsack instance. In fact, the set of feasible solutions can be partitioned
into two subsets: those with $x_{n+1}=1$ and those with $x_{n+1} =0$. If we fix the value of $x_{n+1}$, the residual problem in both cases reduces to a Knapsack instance where all the $n$ remaining objects have the same size. Therefore a greedy algorithm can be used to find an optimal solution to each subproblem. Similarly, from a polyhedral point of view, one easily realizes that the convex hull of the points in $P(n+1,\varepsilon)$ admits a compact extended formulation. This can be obtained as follows: partition again the family of feasible solutions into two sets based on the value of $x_{n+1}$. As one immediately checks, the convex hull of each of those two subproblems has a compact formulation. 
Therefore one can apply Balas' union of polyhedra technique~\cite{Bal} to obtain the claimed compact extended formulation.
However, the situation drastically change when we require our formulation to yield a polytope in $\mathbb R^{n+1}$. In this case, we can not even $\varepsilon$-\emph{approximate} our polytope without using a super-polynomial number of inequalities,
as we are now going to show.

\paragraph{Constructing a family of points.}  Our goal is to define a family of points ${\cal X}(n+1,\varepsilon)$ that together with $P(n+1,\varepsilon)$ satisfy properties (i)-(iii).
The main idea lies in realizing that this can be achieved by constructing a suitable \emph{set system} $\mathcal S$ of subsets of $[n]$, containing a super-polynomial number of sets with the property that the sets have pairwise small intersection.

A way to construct such a set system is following an idea of Nisan and Wigderson~\cite[Lemma 2.5]{NiWi}. Consider the field $\bb F$ with $\sqrt{n}$ elements, which exists because we are assuming $\sqrt n$ to be a prime number. Consider the universe set $U$ made of $n$ elements represented as ordered pairs $(a,b)$ with $a,b \in {\bb F}$. To each polynomial $\pi$ of $\bb F$ of degree at most $ \lfloor \sqrt{n}/2 - 4\rfloor $, we associate the set $S_{\pi}=\{(a,\pi(a))\}_{a \in {\bb F}}$. Note that $|S_{\pi}|=\sqrt{n}$. Let ${\cal S}$ be the family of all such sets. It is well-known (see e.g.~\cite{Intro}) that, for each choice of $2(k+1)$ elements $a_1,\dots,a_{k+1},b_1,\dots,b_{k+1} \in \bb F$, with the first $k+1$ being distinct, there exists exactly one polynomial $\pi$ in $\bb F$ of degree at most $k$ such that $\pi(a_i)=b_i$ for each $i$. We then deduce $|{\cal S}|=n^{\frac{1}{2}(\lfloor \sqrt{n}/2 - 4\rfloor+1)}$ and $|S_{\pi} \cap S_{\pi'}| \leq \lfloor \sqrt{n}/2 - 4\rfloor$ for each pair of polynomials $\pi \neq \pi'$ of $\bb F$. Hence, we obtained a family ${\cal S}\subseteq 2^n$ such that:
\begin{itemize}
\item[(a)] $|S|=\sqrt n$ for each $S \in {\cal S}$,
\item[(b)] $ |S \cap  S'| \leq \lfloor \sqrt{n}/2 - 4\rfloor$ for each $S\neq S' \in {\cal S}$,
\item[(c)] $|\mathcal S| =n^{\frac{1}{2}(\lfloor \sqrt{n}/2 - 4\rfloor+1)}$.
\end{itemize}

\smallskip

From know on, we denote the elements of $U$ with numbers $1$ to $n$, and identify the $i-$th element with the $i$-th object from the knapsack instance, for $i=1,\dots,n$. For each $S \in {\cal S}$ we define a point $x^S \in \mathbb R^{n+1}$ with coordinates:
\begin{itemize}
\item[-] $x^S_{n+1}=\frac{4}{5}$,
\item[-] $x^S_i = 1-\frac{\varepsilon}{2}$ for $i \in S$,
\item[-] $x^S_i = 0$ for $i \notin S$ and $i \neq n+1$.
\end{itemize}

We let ${\cal X}(n +1 ,\varepsilon)$ be the family of these points, that is ${\cal X}(n +1,\varepsilon):=\{x^S : S \in {\cal S}\}$.

\paragraph{Verifying properties.} We now show that ${\cal X}(n+1,\varepsilon)$ and $P(n+1,\varepsilon)$ satisfy properties (i), (ii) and (iii), concluding our proof.
\begin{claim}\label{cl:2}
 ${\cal X}(n + 1,\varepsilon)$ satisfies \emph{(i)}.
\end{claim}
\begin{proof}
$| {\cal X}(n + 1,\varepsilon)| = |\mathcal S|$ and, by (c), $|\mathcal S|= n^{\frac{1}{2}(\lfloor \sqrt{n}/2 - 4\rfloor+1)}$.
\end{proof}

\begin{claim}\label{cl:2}
${\cal X}(n+1,\varepsilon)$ and $P(n+1,\varepsilon)$ satisfy \emph{(ii)}.
\end{claim}

\begin{proof}
Fix $S \in {\cal S}$, and define $c \in \R^n$  as follows:

\begin{itemize}
\item[-] $c_{n+1}=1$,
\item[-] $c_i = \frac{1}{2 \varepsilon \sqrt n}$ for $i \in S$,
\item[-] $c_i = 0$ for $i \notin S$ and $i \neq n+1$.
\end{itemize}

We first show that
$$\max \{cx : x \in P(n+1,\varepsilon)\}= \frac{1}{2\varepsilon}.$$

Let $x^*$ be the optimal solution to the problem $\max \{cx : x \in P(n+1,\varepsilon)\}$. Without loss of generality, we can assume that the non-zero
entries of $x^*$ have indices in the set $S \cup \{n+1 \}$, since all other indices have objective function coefficient equal to zero.

Let's first assume that $x^*_{n+1} =1$.
Then  $$\sum_{i \in S} c_i x^*_i = \sum_{i \in S} \frac{1}{2 \varepsilon \sqrt n} x^*_i \leq \frac{1}{2\varepsilon} -1,$$ where last inequality
is implied by feasibility of $x^*$. It follows that $cx^* \leq \frac{1}{2\varepsilon}$.

If instead $x^*_{n+1} = 0$, 
we get $\sum_{i \in S} c_i x^*_i \leq |S| \frac{1}{2\varepsilon \sqrt n} = \frac{1}{2\varepsilon}$. On the other hand, $$cx^S=(1-\frac{\varepsilon}{2})(\frac{1}{2\varepsilon}) + \frac{4}{5} = \frac{1}{2\varepsilon} - \frac{1}{4} +\frac{4}{5}=\frac{1}{2\varepsilon}+ \frac{11}{20}, \hbox{ hence}$$

$$\label{eq:1}(1-\varepsilon)\frac{cx^S}{ \max \{cx : x \in P(n+1,\varepsilon)\}}\geq (1-\varepsilon)(1 + \frac{11}{10}\varepsilon)=1 + \varepsilon(\frac{1}{10}-\frac{11}{10}\varepsilon) >1$$

where last inequality follows because $\frac{1}{10}-\frac{11}{10}\varepsilon >0$, since $\varepsilon < \frac{1}{11}$ .
\end{proof}

\begin{claim}
For each $S\neq S' \in {\cal S}$,  we have $\frac{1}{2} x^S + \frac{1}{2}x^{S'}  \in P(n+1,\varepsilon)$. Hence, ${\cal X}(n+1,\varepsilon)$ and $P(n+1,\varepsilon)$ satisfy \emph{(iii)}.
\end{claim}
\begin{proof}
In the following, we abbreviate $P(n+1,\varepsilon)$ by $P$. Let $x^{SS'}:= \frac{1}{2} x^S + \frac{1}{2}x^{S'}$. We show that $x^{SS'} \in P$ by showing that $x^{SS'} \leq \bar x$ (where the inequality is coordinate-wise) for a vector $\bar x \in P$, that is, $x^{SS'}$ is \emph{dominated} by a point $\bar x$ inside the Knapsack polytope. Since the Knapsack polytope is down-monotone, the claim follows.

We now describe how to choose $\bar x$. Let $\alpha:=|S\cap S'|\leq \sqrt{n}/2-4$ and $\beta:= \lfloor (1 - 2 \varepsilon) \sqrt n \rfloor$. By our choice of $n$ and $\epsilon$, we have $\beta-\alpha > 0$.
Consider the family ${\cal K}$ all possible subsets of $[n+1]$ of elements which:
\begin{itemize}
\item contain $\{n+1\} \cup (S\cap S')$;
 \item contain $\beta -\alpha$ elements from $S\triangle S'$.
  \end{itemize}
  Note that each of those sets has weight  (with respect to our definition of $P(n+1, \varepsilon)$)
  $$n + \frac{\beta}{2\varepsilon \sqrt n}  \leq n + \frac{(1 - 2 \varepsilon) \sqrt n}{2 \varepsilon \sqrt n}   \leq n+\frac{1}{2\varepsilon}-1$$ hence it is a feasible Knapsack solution. Therefore, the incidence vector of each of those Knapsack solutions belongs to $P$, and we associate to each of them the multiplier $\frac{4}{5} \cdot \frac{1}{|{\cal K}|}$. Moreover, consider the set $S \cup S'$. This set also induces a feasible solution of $P$, since the total weight of the elements is at most $1/\varepsilon<n$, and therefore its incident vector is a point in $P$. We associate to this incident vector the multiplier $1/5$. By construction, those multipliers induce a convex combination of points from $P$, hence a point $\bar x \in P$.

 \smallskip
 We are left with showing that $\bar x$ dominates $x^{SS'}$. Clearly, for any index $i \notin S \cup S' $ we have $ \bar{x}_i = 0 = x_i^{SS'} $. For the index $n+1$
we have $\bar x_{n+1}=x^{SS'}_{n+1}=4/5$. For indices $i \in S \cap S'$ we have $\bar x_i=1> (1 - \frac{\varepsilon}{2}) =x^{SS'}_i$. 
 Therefore, it only remains to analyze indices $i \in S \Delta S'$.

 \smallskip
 Observe that $|S\triangle S'|= 2 \sqrt n- 2\alpha$, and hence that

$$\sum_{i \in  S \triangle S'}\bar x_i = \frac{4}{5} (\beta -\alpha)+ \frac{1}{5} (2 \sqrt n-2\alpha) \geq \frac{4}{5} ((1 - 2 \varepsilon) \sqrt n -1-\alpha)+ \frac{1}{5} (2 \sqrt n-2\alpha).$$ Since $\bar x_i=\bar x_j$ for $i,j \in S \triangle S'$, for each $i \in S \triangle S'$, one has:

$$\bar x_i \geq \frac{4}{5} \cdot \frac{(1 - 2 \varepsilon) \sqrt n -1-\alpha}{2 \sqrt n-2\alpha} + \frac{1}{5}$$
and we want to show that the latter value exceeds $x^{SS'}_i$, that is, we would like the following inequality to hold

$$\frac{4}{5} \cdot \frac{(1 - 2 \varepsilon) \sqrt n -1-\alpha}{2 \sqrt n-2\alpha} + \frac{1}{5} \geq  \frac{1}{2}-\frac{\varepsilon}{4}$$

Now notice that
\begin{eqnarray*}\frac{4}{5} \cdot \frac{(1 - 2 \varepsilon) \sqrt n -1-\alpha}{2 \sqrt n-2\alpha} + \frac{1}{5}& \geq & \frac{1}{2}-\frac{\varepsilon}{4}\\
&\Leftrightarrow&\\
(1 - 2 \varepsilon) \sqrt n -1-\alpha & \geq & \Big(\frac{3}{8} -\frac{5}{16} \varepsilon\Big) \Big(2 \sqrt n-2\alpha\Big)\\
&\Leftrightarrow&\\
 \frac{\frac{1}{4} -\frac{11}{8} \varepsilon }{ \frac{1}{4} + \frac{5}{8} \varepsilon} \sqrt n & \geq & \alpha + \frac{1}{\frac{1}{4}+ \frac{5}{8}\epsilon}
\end{eqnarray*}

and the latter inequality is true since $0 < \varepsilon < \frac{2}{27}$ implies that
$$ \frac{\frac{1}{4} -\frac{11}{8} \varepsilon }{ \frac{1}{4} + \frac{5}{8} \varepsilon} \geq \frac{1}{2}  \; \mbox{  and } \; 4 >  \frac{1}{\frac{1}{4}+ \frac{5}{8}\epsilon},$$
therefore, recalling that $\alpha \leq   \sqrt{n}/2 - 4$, we get

$$\frac{\frac{1}{4} -\frac{11}{8} \varepsilon }{ \frac{1}{4} + \frac{5}{8} \varepsilon} \sqrt n \geq  \frac{1}{2} \sqrt n \geq  \alpha +4 \geq  \alpha + \frac{1}{\frac{1}{4}+ \frac{5}{8}\epsilon}$$
as claimed.

\end{proof}

\section{An upper bound for down-monotone polytopes}

In this section we prove Theorem \ref{thm:upper_bound}.

Let $P \subseteq \R^n_{+}$ be a down-monotone polytope. By scaling it, we can suppose that $P\subseteq [0,1]^n$. For any $\varepsilon$ with $1/2\geq\varepsilon >0$, set $\gamma :=\varepsilon/(1-\varepsilon)\leq1$. We will approximate $P$ with inequalities whose coefficients belong to the following set:

$$S:=\{0\}\cup \{\lfloor(1 +\frac{\gamma}{2})^{\ell}\rfloor : \ell =0,1,\dots,\lfloor \log_{(1+\frac{\gamma}{2})} \Big(\frac{5n}{\gamma}\Big)\rfloor \}.$$

Note that $|S| \leq 1+\min\{\lfloor \log_{(1+\frac{\gamma}{2})}\Big(\frac{5n}{\gamma}\Big) \rfloor, \frac{5n}{\gamma}\}$, and we have

$$\log_{(1+\frac{\gamma}{2})} \left(\frac{5n}{\gamma}\right)=O\left(\frac{\log (5n/\epsilon)}{\log(1+\frac{\epsilon}{2})}\right) = O\left(\frac{\log n/\epsilon}{\epsilon}\right),$$

where the last equality follows from the fact that $\frac{\log(1+ \frac{\epsilon}{2})}{\epsilon}=\Theta(1)$ for $\epsilon \rightarrow 0$, which can be verified using e.g. l'H\^{o}pital's rule. We conclude $|S|=O(\min\{\frac{\log (n/\epsilon)}{\epsilon},\frac{n}{\epsilon}\})$, and in
particular $|S|=O(\log n)$ for any constant $\varepsilon$.

\smallskip

Let now $Q$ be the polytope obtained by intersecting $1\geq x\geq 0$ with all the inequalities of the form $\tilde c x \leq \max\{\tilde cy : y \in P\} $ with $\tilde c_i \in S$ for $i \in [n]$. That is:
$$ Q := \{x \in \mathbb R^n: \; 0 \leq x \leq 1,  \; \tilde c x \leq \max\{\tilde cy : y \in P\} \mbox{ for all } \tilde c \in \mathbb R^n \mbox{ with } \tilde c_i \in S \; \forall i \in [n]  \}.$$

In order to prove Theorem \ref{thm:upper_bound}, we are left to show the following.
\begin{lemma}\label{lem:ub}
$Q$ is an $\epsilon$-approximated formulation for $P$.
\end{lemma}

\begin{proof}
Clearly $P \subseteq Q$. Consider any vector $c \in \mathbb R^n$. We have to show that
$$\beta := \max \{cx: x \in P\} \geq (1-\varepsilon) \max \{cx: x \in Q\}.$$
As $P$ is down-monotone, we can assume $c\geq 0$.
Define $\tilde c$ as follows.
Let $K=||c||_\infty \frac{\gamma}{5n}$ and $G:= \{i\in [n]: c_i < K\}$. Set $\tilde c_i = 0$ for $i \in G$ and $\tilde c_i =\lfloor(1+ \frac{\gamma}{2})^{\lfloor \log_{(1+\frac{\gamma}{2})}(\frac{c_i}{K}) \rfloor}\rfloor$ otherwise. Note that we have $\tilde c_i \in S$ for all $i \in [n]$.

Let $\tilde \beta := \frac{\beta}{K}$. We now show
\begin{equation*}
\max\{c x:   x \in Q \} \leq \max\{c x:  \tilde c x \leq \tilde \beta, 0\leq x\leq 1 \}\leq (1+\gamma)\beta=\beta /(1-\epsilon),
\end{equation*} which immediately concludes the proof. The equality follows by definition of $\gamma$. The first inequality follows easily by observing that the constraint
$\tilde c x \leq \tilde \beta$ is a valid constraint for $Q$, as $\tilde c_i \in S $  for all $ i \in [n]$ and for $x\in P$ one has $$\sum_{i} \tilde c_i x_i = \sum_{i \notin G} \tilde c_i x_i \leq \sum_{i \notin G} (1+ \frac{\gamma}{2})^{ \log_{(1+\frac{\gamma}{2})}(\frac{c_i}{K}) } x_i \leq \frac{1}{K} \sum_{i} c_i x_i \leq \frac{\beta}{K}=\tilde \beta.$$

For the second inequality, let $x^* \in [0,1]^n$ be such that $c x^*> (1+\gamma)\beta$. We show that $\tilde c x^* > \tilde \beta$, concluding the proof. We have:

\begin{IEEEeqnarray}{lll}\tilde c x^* \quad & = & \quad \sum_{i \notin G} \tilde c_i x^*_i  \nonumber \\
& \geq & \quad \sum_{i \notin G} \frac{2}{2+\gamma} \frac{c_i}{K} x^*_i - n  \label{eq:serie-b} \\
& = &\quad \frac{2}{2+\gamma}\frac{1}{K} \Big ( \sum_{i \in [n]} c_i x^*_i - \sum_{i \in G} c_i x^*_i\Big) - n  \nonumber\\
& \geq & \quad \frac{2}{2+\gamma}\frac{1}{K}cx^* - \frac{2}{2+ \gamma}\frac{1}{K} Kn - n \label{eq:los'anna} \\
& > & \quad \frac{2}{2+\gamma}\frac{1}{K}\beta(1+\gamma) - \frac{2n}{2+ \gamma} -n  \nonumber\\
& = & \quad  \tilde \beta + \frac{\gamma}{2+\gamma} \tilde \beta - \frac{2n}{2+\gamma} -n  \nonumber\\
& \geq & \quad \tilde \beta,  \label{eq:the-end-is-the-beginning-is-the-end}\end{IEEEeqnarray}

\noindent where \eqref{eq:serie-b} is implied by $x^* \in [0,1]^n$ and, for $i \notin G$,  $$\tilde c_i \geq (1+ \frac{\gamma}{2})^{ \log_{(1+\frac{\gamma}{2})}(\frac{c_i}{K})  - 1}-1 = \frac{c_i}{K} (1+ \frac{\gamma}{2})^{-1} -1= \frac{2}{2+\gamma} \frac{c_i}{K}-1;$$
\eqref{eq:los'anna} follows from the fact that $c_i < K$ for $i \in G$, $|G|\leq n$, and $x^* \leq 1$;  and \eqref{eq:the-end-is-the-beginning-is-the-end} holds because $$\frac{\gamma}{2+\gamma} \tilde \beta - \frac{2n}{2+\gamma} - n\geq 0 \Leftrightarrow \tilde \beta \geq \frac{4n + \gamma n}{\gamma} \Leftrightarrow  \beta \geq Kn\frac{(4+\gamma)}{\gamma} = ||c||_{\infty}\frac{4+\gamma}{5},$$ and the latter is true because our choices of parameters imply $\beta \geq ||c||_\infty$ and $\gamma \leq 1$.
\end{proof}

\bigskip

\noindent {\bf Acknowledgments.} We thank Carsten Moldenhauer for useful discussions.

\end{document}